\documentclass{amsart}

\usepackage{amsmath,amssymb}
\usepackage{mathrsfs}
\usepackage{mathtools}
\usepackage{enumerate}
\usepackage{tikz-cd}
\usepackage[all]{xy}
\usepackage{aliascnt}
\usepackage[colorlinks=true, linkcolor=blue, citecolor=blue]{hyperref}
\usepackage{ytableau}
\usepackage{graphicx}
\usepackage{enumitem}

\usepackage{fullpage}

\newtheorem{theorem}{Theorem}

\newaliascnt{lemma}{theorem}
\newtheorem{lemma}[lemma]{Lemma}
\aliascntresetthe{lemma}

\newaliascnt{corollary}{theorem}
\newtheorem{corollary}[corollary]{Corollary}
\aliascntresetthe{corollary}

\newaliascnt{proposition}{theorem}
\newtheorem{proposition}[proposition]{Proposition}
\aliascntresetthe{proposition}

\newaliascnt{conjecture}{theorem}

\aliascntresetthe{conjecture}

\newaliascnt{question}{theorem}

\aliascntresetthe{question}

\theoremstyle{definition}

\newaliascnt{definition}{theorem}
\newtheorem{definition}[definition]{Definition}
\aliascntresetthe{definition}

\newaliascnt{remark}{theorem}
\newtheorem{remark}[remark]{Remark}
\aliascntresetthe{remark}

\newaliascnt{example}{theorem}
\newtheorem{example}[example]{Example}
\aliascntresetthe{example}

\newaliascnt{notation}{theorem}

\aliascntresetthe{notation}

\newtheorem*{acknowledgements}{Acknowledgements}

\newif\ifhascomments \hascommentstrue
\ifhascomments
  \newcommand{\matt}[1]{{\color{red}[[\ensuremath{\spadesuit\spadesuit\spadesuit} #1]]}}
  \newcommand{\yohan}[1]{{\color{red}[[\ensuremath{\clubsuit\clubsuit\clubsuit} #1]]}}
  \newcommand{\ron}[1]{{\color{red}[[\ensuremath{\heartsuit\heartsuit\heartsuit} #1]]}}
\else
  \newcommand{\ron}[1]{}
  \newcommand{\yohan}[1]{}
  \newcommand{\matt}[1]{}
\fi

\renewcommand{\setminus}{\smallsetminus}

\newcommand{\Z}{\mathbb{Z}}
\newcommand{\N}{\mathbb{N}}

\newcommand{\cF}{\mathcal{F}}

\newcommand{\cT}{\mathcal{T}}

\newcommand{\cA}{\mathcal{A}}

\newcommand{\fm}{\mathfrak{m}}

\newcommand{\Ann}{\mathrm{Ann}}

\DeclareMathOperator{\spann}{Span}
\DeclareMathOperator{\supp}{supp}

\DeclareMathOperator{\MM}{M}

\newcounter{casecounter}
\tikzset{cong/.style={draw=none,edge node={node [sloped, allow upside down, auto=false]{$\cong$}}},
         Isom/.style={above,every to/.append style={edge node={node [sloped, allow upside down, auto=false]{$\sim$}}}}}

\title{On the algebra generated by three commuting matrices:~combinatorial cases}

\author{Ron Cherny, Matthew Satriano, and Yohan Song}

\thanks{MS was partially supported by a Discovery Grant from the National Science and Engineering Research Council of Canada as well as a Mathematics Faculty Research Chair from the University of Waterloo}

\address{Ron Cherny, Department of Pure Mathematics, University of Waterloo}
\email{rcherny@uwaterloo.ca}

\address{Matthew Satriano, Department of Pure Mathematics, University of Waterloo}
\email{msatriano@uwaterloo.ca}

\address{Yohan Song, Department of Pure Mathematics, University of Waterloo}
\email{y343song@uwaterloo.ca}

\DeclareMathOperator{\score}{score}
\DeclareMathOperator{\height}{height}

\begin{document}

\begin{abstract}
Gerstenhaber proved in 1961 that the unital algebra generated by a pair of commuting $d\times d$ matrices over a field has dimension at most $d$. It is an open problem whether the analogous statement is true for triples of matrices which pairwise commute. We answer this question for special classes of triples of matrices arising from combinatorial data.
\end{abstract}

\maketitle

\numberwithin{theorem}{section}
\numberwithin{lemma}{section}
\numberwithin{corollary}{section}
\numberwithin{proposition}{section}
\numberwithin{conjecture}{section}
\numberwithin{question}{section}
\numberwithin{remark}{section}
\numberwithin{definition}{section}
\numberwithin{example}{section}
\numberwithin{notation}{section}

\setcounter{tocdepth}{1}
\tableofcontents

\section{Introduction}
\label{Sec:Introduction}

Given a field $k$, the well-known Cayley--Hamilton theorem asserts that every matrix $A\in\MM_d(k)$ is a root of its characteristic polynomial. In particular, the (unital) algebra generated by $A$ is a $k$-vector space of dimension at most $d$. In 1961, Gerstenhaber \cite{Gerstenhaber} generalized this result to pairs of commuting matrices:~if $A,B\in\MM_d(k)$ are commuting $d\times d$ matrices then the algebra they generate has dimension at most $d$. In contrast, it is known that for all $n\geq4$, there exists $d$ and pairwise commuting matrices $A_1,\dots,A_n\in\MM_d(k)$ such that the algebra generated by the $A_i$ has dimension strictly larger than $d$, see Example \ref{ex:4varcounter-example}. It has been a longstanding open question (which we refer to as the \emph{Gerstenhaber problem}) to determine whether pairwise commuting matrices $A,B,C\in\MM_d(k)$ generate an algebra of dimension at most $d$.

Let us begin by discussing some of the known cases of the Gerstenhaber problem. To begin, Gerstenhaber's original proof was algebro-geometric. He considered the algebraic variety $C(2,d)$ parameterizing pairs $(A,B)$ of commuting $d\times d$ matrices and proved it is irreducible (cf.~\cite{MTT}). This allowed him to reduce to the case of \emph{generic} pairs $(A,B)$ of commuting matrices, which are simultaneously diagonalizable; hence, the result follows from Cayley--Hamilton. In fact, Gerstenhaber's technique of reducing to simultaneously diagonalizable matrices works whenever we have irreducibility of the variety $C(n,d)$, parameterizing pairwise commuting $d\times d$ matrices $(A_1,\dots,A_n)$. Unfortunately, for $n\geq3$, $C(n,d)$ is notoriously complicated. For $n\geq4$ and $d\geq4$, $C(n,d)$ has multiple irreducible components, see \cite{Gerstenhaber,Guralnick92}. For $n=3$, much less is known:~the variety $C(n,d)$ is irreducible for $d\leq 10$ \cite{Sivic2} and reducible for $d\geq 29$ \cite{HolOmla,NgoSivic14}. See also \cite{JelisiejewSivic22} for further results on the structure of components of $C(n,d)$. In general, the Gerstenhaber problem is reduced to checking at the generic points of every irreducible component of $C(n,d)$; however such an approach is essentially intractable.

The aforementioned results are geometric and concern the structure of $C(n,d)$. Other proofs of Gerstenhaber's theorem were later discovered, such as commutative algebraic proofs \cite{Wadsworth, Bergman} and linear algebraic proofs \cite{BH, LL}. In addition to the case $d\leq10$ mentioned above, several other cases of the Gerstenhaber problem are known when one imposes linear algebraic constraints. For example, if one of the three $d\times d$ matrices $A_1,A_2,A_3$ has nullity at most $3$, then it was shown in \cite{GS,sivicII} that the algebra the matrices generate has dimension at most $d$. The Gerstenhaber problem is also known if one of the matrices has index at most $2$, i.e., some $A_i^2=0$, see \cite{HolOmla}. We refer to \cite{SethurSurvey, HolbrookOmeara} for a survey of further results.

\vspace{1.2em}

In this paper, the viewpoint we take is to break up the Gerstenhaber problem based on the minimal number of \emph{generating vectors} required. Given pairwise commuting $d\times d$ matrices $(A_1,\dots,A_n)$, let $\cA$ be the algebra they generate. We say $v_1,\dots,v_r\in k^d$ are generating vectors if
\[
\spann\{Av_j\mid A\in\cA,\,j\leq r\}=k^d.
\]
When $r=1$, it is straightforward to check that $\dim_k\cA\leq d$. 
On the other hand, for $r=2$, the Gerstenhaber problem is still open and highly non-trivial. Indeed, the simplest subcase when $r=2$ is when $A_iv_2\in\spann\{Av_1\mid A\in\cA,\,j\leq r\}$ for all $i$; this case was only recently resolved by Rajchgot and the second author \cite[Theorem 1.5]{RajSat}, showing $\dim_k\cA\leq d$.

Following \cite{RSS}, our current paper considers a broad class of combinatorially motivated examples when $r=2$. Before giving the formal definition, we begin with an example.

\begin{example}\label{ex:4varcounter-example}
As mentioned above, there exist choices of pairwise commuting matrices $A_1,\dots,A_4\in\MM_d(k)$ such that the algebra $\cA$ generated by the $A_i$ has $\dim_k\cA>d$. 
The standard such example is given by taking $d=4$ and letting the four pairwise commuting matrices be $E_{13}$, $E_{23}$, $E_{14}$, and $E_{24}$. Then $\cA$ has a basis given by these matrices as well as the identity matrix $I$, hence $\dim_k\cA=5>4=d$.

This example can be understood combinatorially as follows. Let $S=k[x_1,\dots,x_4]$ and consider the monomial ideals
\[
I=(x_1,x_2)^2+(x_3,x_4)\quad\textrm{and}\quad J=(x_3,x_4)^2+(x_1,x_2).
\]
We consider the $S$-module $M$ obtained from $S/I\oplus S/J$ by gluing $(x_1,0)$ to $(0,x_3)$, and gluing $(x_2,0)$ to $(0,x_4)$, i.e., 
\[
M=(S/I\oplus S/J)/\langle (x_1,0)-(0,x_3), (x_2,0)-(0,x_4)\rangle.
\]
Then $M$ is a vector space of dimension $d=4$ with basis $(1,0)$, $(0,1)$, $(x_1,0)$, and $(x_2,0)$. Multiplication by $x_i$ on $M$ yields $n=4$ commuting matrices. These matrices are precisely the same as the standard example given in the previous paragraph. $\diamond$
\end{example}

With this as motivation, we now define our combinatorial matrices, cf.~\cite[\S4.2]{RSS}.

\begin{definition}\label{def:combinatorial-mats}
    Let $S=k[x_1,\dots,x_n]$. Let $I\subset K\subset S$ and $J\subset L\subset S$ be monomial ideals with $\dim_k S/I<\infty$ and $\dim_k S/J<\infty$. Given an isomorphism $\phi\colon K/I\xrightarrow{\simeq} L/J$ of $S$-modules sending monomials to monomials, we obtain an $S$-module
    \[
    M=(S/I\oplus S/J)/\langle\,(f,-\phi(f))\mid f\in S/I\,\rangle
    \]
    equipped with the natural monomial basis. Letting $d=\dim M$ and $A_i$ be the $d\times d$ matrix given by the linear map $M\xrightarrow{\cdot x_i}M$, we say $(A_1,\dots,A_n)$ are \emph{associated to} $(I,J,K,L,\phi)$.
\end{definition}

\begin{remark}\label{rmk:combinatorial-examples-have-r=2}
    With notation as in Definition \ref{def:combinatorial-mats}, by construction $(A_1,\dots,A_n)$ pairwise commute and require at most two generating vectors, namely $(1,0), (0,1)\in M$. In Section \ref{sec:review-combinatorial-approach}, we review how to think of such $(A_1,\dots,A_n)$ as coming from $n$-dimensional partition shapes.
\end{remark}

We wish to emphasize that, despite the simplicity of this definition, this case is already non-trivial. 
Indeed, letting $\fm:=(x_1,\dots,x_n)$, note that in Example \ref{ex:4varcounter-example}, we have $K/I\simeq(S/\fm)^{\oplus 2}$. Thus, even if one restricts attention to examples where $K/I$ is as simple as possible, namely $(S/\fm)^{\oplus c}$, we already obtain a broad enough class to encompass the standard example of four pairwise commuting $d\times d$ matrices $(A_1,\dots,A_4)$ with $\dim_k\cA>d$.




\begin{definition}\label{def:Gerstenhaber-problem-gluing}
    Let $S=k[x_1,\dots,x_n]$ and $N$ be an $S$-module. 
    We say the \emph{combinatorial Gerstenhaber problem holds when gluing along} $N$ if for all $(I,J,K,L,\phi)$ as in Definition \ref{def:combinatorial-mats} with $K/I\simeq N$, we have
    \[
    \dim_k\cA\leq\dim M,
    \]
    where $(A_1,\dots,A_n)$ is associated to $(I,J,K,L,\phi)$ and $\cA$ is the algebra generated by $(A_1,\dots,A_n)$.
\end{definition}

In this paper, we prove:

\begin{theorem}\label{thm:main-intro}
    If $S=k[x_1,x_2,x_3]$ and $N=\bigoplus_i S/(x_1,x_2,x_3^{n_i})$, then the combinatorial Gerstenhaber problem holds when gluing along $N$.
\end{theorem}

Theorem \ref{thm:main-intro} generalizes \cite[Theorem 4]{RSS}, where the result was shown when all $n_i=1$. The proof given in \cite{RSS}, however, does not lend itself to generalization and a new idea was required. 
Our proof of Theorem \ref{thm:main-intro} uses a series of reductions to turn this three-dimensional problem into a two-dimensional one using objects that we call floor plans. We ultimately prove the main theorem by constraining the shape that such floor plans can assume.

\begin{acknowledgements}
This paper is the outcome of an NSERC-USRA project; we thank NSERC for their support. 
This project benefited greatly from conversations with Oliver Pechenik and Jenna Rajchgot; it is a pleasure to thank them. The second author would especially like to thank his son, Quinn, for letting him use his toy blocks, which helped with three-dimensional visualization.
\end{acknowledgements}

\section{Combinatorial Gerstenhaber problem and Young diagrams}
\label{sec:review-combinatorial-approach}

In this section, we reduce Theorem \ref{thm:main-intro} to a problem in combinatorics.
Much of the material in this section is based on \cite[\S4.1--4.2]{RSS}.

To begin, the monomials in $S:=k[x_1,\dots,x_n]$ are in bijection with elements of $\N^n$ by identifying $a=(a_1,\dots,a_n)\in\N^n$ with $x^a:=x_1^{a_1}\cdots x_n^{a_n}$. Recall that an \emph{$n$-dimensional Young diagram} (also known as a \emph{standard set} or \emph{staircase diagram}) is a finite subset $\lambda\subset\N^n$ such that for all $v,w\in\N^n$ with $v\leq w$ (in the standard partial order on $\N^n$), if $w\in\lambda$ then $v\in\lambda$. 
Given a monomial ideal $I\subset S$ with $\dim_k S/I<\infty$, we obtain an $n$-dimensional Young diagram $\lambda\subset\N^n$ given by the set of $a\in\N^n$ with $x^a\notin I$; moreover, this yields an inclusion-reversing bijection between such monomial ideals and $n$-dimensional Young diagrams, see, e.g., \cite[Chapter 3]{MS} for further details.

Next, if $I$ is as above and $I\subset K$ with $K$ a monomial ideal, then let $\nu\subset\N^n$ be the set of $a\in\N^n$ with $x^a\in K/I$. We see $\nu=\lambda\setminus\lambda'$ where $\lambda$ (resp.~$\lambda'$) is the $n$-dimensional Young diagram associated to $I$ (resp.~$K$). Sets obtained as the difference of two $n$-dimensional Young diagrams are referred to as \emph{skew shapes}. We say a skew shape $\epsilon$ is \emph{connected} if for any $a,b\in\epsilon$, there exist $a',b'\in\N^n$ such that $a+a'=b+b'\in\epsilon$. Every skew shape can be written uniquely as a disjoint union of connected ones, which we refer to as its \emph{connected components}.

\begin{lemma}\label{l:skew-shape-monomial-iso}
    Let $S=k[x_1,\dots,x_n]$. Let $I\subset K\subset S$ and $J\subset L\subset S$ be monomial ideals with $\dim_k S/I<\infty$ and $\dim_k S/J<\infty$. Let $\nu$ (resp.~$\epsilon$) be the skew shape associated with $K/I$ (resp.~$L/J$). Let $\nu=\nu_1\amalg\dots\amalg\nu_m$ and $\epsilon=\epsilon_1\amalg\dots\amalg\epsilon_\ell$ be the decompositions into connected components.

    Every isomorphism $\phi\colon K/I\xrightarrow{\simeq} L/J$ of $S$-modules sending monomials to monomials is given as follows. We have a bijection $\sigma\colon\{1,\dots,m\}\to\{1,\dots,\ell\}$ and an element $c_i\in\Z^n$ such that $\nu_i+c_i=\epsilon_{\sigma(i)}$. The map $\phi$ is then given by $\phi(x^a)=x^{a+c_i}$ for all $a\in\nu_i$.
\end{lemma}
\begin{proof}
Since $\phi$ is an $S$-module isomorphism and induces a bijection from monomials of $K/I$ to monomials of $L/J$, we see $\phi(x^a)=x^{\psi(a)}$, where $\psi\colon\nu\to\epsilon$ is an isomorphism of posets.

Thus, we need only prove that if $a$ and $b$ are in the same connected component of $\nu$, then $\psi(a)$ and $\psi(b)$ are in the same connected component of $\epsilon$. To see this, say $a,b\in\nu_i$ and let $a',b'\in\N^n$ such that $a+a'=b+b'\in\nu_i$. Then
\[
x^{a'+\psi(a)}=x^{a'}\phi(x^a)=\phi(x^{a+a'})=\phi(x^{b+b'})=x^{b'}\phi(x^b)=x^{b'+\psi(b)}
\]
in $L/J$. Since $\phi$ is an isomorphism and $a+a'\in\nu_i$, we know $\phi(x^{a+a'})\neq0$. Thus, $a'+\psi(a)=b'+\psi(b)$. Since $x^{a'+\psi(a)}=\phi(x^{a+a'})\neq0$, we see $a'+\psi(a)$ is in the same connected component as $\psi(a)$; similarly, $b'+\psi(b)$ is in the same connected component as $\psi(b)$. It follows that $\psi(a)$ and $\psi(b)$ are in the same connected component of $\epsilon$.
\end{proof}

In light of Lemma \ref{l:skew-shape-monomial-iso}, we say two skew shapes $\nu,\epsilon\subset\N^n$ are \emph{translationally equivalent} if there exists $w\in\Z^n$ such that $\nu+w=\epsilon$. Note that this defines an equivalence relation on skew shapes. Note that each skew shape $\nu$ has a unique lex-smallest point $\ell_\nu$. We may therefore normalize our skew shapes by considering $\nu-\ell_\nu$; we refer to this translated skew shape as an \emph{abstract skew shape}. Thus, every skew shape is of the form $\alpha+w$ where $\alpha$ is an abstract skew shape, $w\in\Z^n$, and $\alpha+w\subset\N^n$. Note that we have natural notions of connectedness and connected components for abstract skew shapes.

We can now reinterpret the data in Definition \ref{def:combinatorial-mats} purely combinatorially.

\begin{corollary}\label{cor:combinatorial-mats-Young-diagrams}
Let $S=k[x_1,\dots,x_n]$. Giving $(I,J,K,L,\phi)$ as in Definition \ref{def:combinatorial-mats} is equivalent to giving
\begin{enumerate}
    \item\label{cor:combinatorial-mats-Young-diagrams::datastart} $n$-dimensional Young diagrams $\lambda$ and $\mu$,
    \item connected abstract skew shapes $\nu_1,\dots,\nu_r$, and
    \item\label{cor:combinatorial-mats-Young-diagrams::dataend} lattice points $b_1,\dots,b_r,c_1,\dots,c_r\in\Z^n$
\end{enumerate}
such that
\begin{enumerate}[label=(\alph*)]
    \item\label{cor:combinatorial-mats-Young-diagrams::propertystart}$\nu_1+b_1,\dots,\nu_r+b_r$ are disjoint and contained in $\lambda$,
    \item\label{cor:combinatorial-mats-Young-diagrams::propertyideal} for all $w\in\N^n$ and all $v_i\in\nu_i$ such that $v_i+b_i+w\in\lambda$, we have $v_i+w\in\nu_i$,
    \item $\nu_1+c_1,\dots,\nu_r+c_r$ are disjoint and contained in $\mu$, and
    \item\label{cor:combinatorial-mats-Young-diagrams::propertyend} for all $w\in\N^n$ and all $v_i\in\nu_i$ such that $v_i+c_i+w\in\mu$, we have $v_i+w\in\nu_i$.
\end{enumerate}
Moreover, under this equivalence, $\lambda$ (resp.~$\mu$) is the $n$-dimensional Young diagram corresponding to $I$ (resp.~$J$), and $\nu_i+b_i$ (resp.~$\nu_i+c_i$) are the connected components of the skew shape associated to $K/I$ (resp.~$L/J$).
\end{corollary}
\begin{proof}
    If $(I,J,K,L,\phi)$ are as in Definition \ref{def:combinatorial-mats}, let $\lambda$ (resp.~$\mu$) be the $n$-dimensional Young diagram corresponding to $I$ (resp.~$J$). Let $\beta$ (resp.~$\gamma$) be the skew shape corresponding to $K/I$ (resp.~$L/J$). By Lemma \ref{l:skew-shape-monomial-iso}, we may order the connected components $\beta_1,\dots,\beta_r$ of $\beta$ and $\gamma_1,\dots,\gamma_r$ of $\gamma$ such that $\beta_i$ and $\gamma_i$ are translationally equivalent. Letting $\nu_i$ be the abstract skew shape associated to $\beta_i$, we then see there are $b_1,\dots,b_r,c_1,\dots,c_r\in\Z^n$ such that $\beta_i=\nu_i+b_i$ and $\gamma_i=\nu_i+c_i$. Lastly, note that if $v_i\in\nu_i$, $w\in\N^n$, and $v_i+b_i+w\in\lambda$, then $x^{v_i+b_i+w}$ is not in $I$ hence also not in $K$, i.e., $v_i+b_i+w\notin\nu_i+b_i$, so $v_i+w\notin\nu_i$.

    Conversely, given data in (\ref{cor:combinatorial-mats-Young-diagrams::datastart})--(\ref{cor:combinatorial-mats-Young-diagrams::dataend}) satisfying properties \ref{cor:combinatorial-mats-Young-diagrams::propertystart}--\ref{cor:combinatorial-mats-Young-diagrams::propertyend}, let $I$ (resp.~$J$) be the monomial ideal corresponding to $\lambda$ (resp.~$\mu$). Let $K'$ be the set of monomials $x^a$ such that $a\in\nu_i+b_i$ for some $i$; let $K''$ be linear combinations of the monomials in $K'$. Then $K:=K''+I$ is a monomial ideal containing $I$, and property \ref{cor:combinatorial-mats-Young-diagrams::propertyideal} shows that $K'$ forms a monomial basis for $K/I$. Thus, $\coprod_i(\nu_i+b_i)$ is the skew shape corresponding to $K/I$ with connected components $\nu_i+b_i$. Similarly, we obtain a monomial ideal $L\supset J$ whose such that the skew shape corresponding to $L/J$ has connected components $\nu_i+c_i$. The lattice points $c_i-b_i\in\Z^n$ then yield an $S$-module isomorphism $\phi\colon K/I\xrightarrow{\simeq} L/J$ as in Lemma \ref{l:skew-shape-monomial-iso}.
\end{proof}

\begin{example}
\label{Ex:2d-pair}
Consider the following example (a variant on \cite[Example 4]{RSS}). We have $2$-dimensional Young diagrams (i.e., partitions) $\lambda$ and $\mu$ corresponding respectively to the ideals $I=(x^5,x^4y,x^2y^3,xy^4,y^5)$ and $J=(x^7,x^6y,x^3y^3,x^2y^4,y^5)$.
\begin{center}
    $\lambda =$ \begin{ytableau}
    *(lightgray)\\
    *(lightgray) & *(lightgray)\\
    \; & & *(lightgray) & *(lightgray)\\
    \; & & & *(lightgray)\\
    \; & & & & *(lightgray)
    \end{ytableau}
    $\mu =$ \begin{ytableau}
    \; & *(lightgray)\\
    \; & *(lightgray) & *(lightgray)\\
    \; & & & & *(lightgray)& *(lightgray)\\
    \; & & & & & *(lightgray)\\
    \; & & & & & & *(lightgray)
    \end{ytableau}
\end{center}
The grey regions represent the connected skew shapes $\nu_1,\nu_2,\nu_3$. These abstract skew shapes are depicted as follows.
\begin{center}
    $\nu_1,\nu_2,\nu_3 =$ 
    \begin{ytableau}
        \; \\ \; & 
    \end{ytableau}\ ,
    \begin{ytableau}
        \; &  \\ \none &
    \end{ytableau}\ ,
    \begin{ytableau}
        \;
    \end{ytableau}.
\end{center}
Here $K=(y^3,x^2y^2,x^3y,x^4)$ and $L=(xy^3,x^4y^2,x^5y,x^6)$. $\diamond$
\end{example}

Having now understood Definition \ref{def:combinatorial-mats} combinatorially, the next result gives a combinatorial reinterpretation of the Gerstenhaber problem, cf.~\cite[Equation (6)]{RSS}.

\begin{proposition}\label{prop:combinatorial-Gerstenhaber-Young-diagram}
Keep the notation of Definition \ref{def:combinatorial-mats} and let $\cA$ be the algebra generated by $(A_1,\dots,A_n)$. Let $\lambda$ (resp.~$\mu$) be the $n$-dimensional Young diagram corresponding to $I$ (resp.~$J$). Let $\nu$ be the skew shape associated to $K/I$. Then
\[
\dim_k\cA\leq d\quad\Longleftrightarrow\quad |\nu|\leq |\lambda\cap\mu|.
\]  
\end{proposition}
\begin{proof}
Let $M$ be as in Definition \ref{def:combinatorial-mats}. By \cite[Proposition 1]{RSS}, $\dim_k\cA\leq d$ if and only if $\dim_k S/\Ann(M)\leq d$. Lemma 8 and the paragraph afterwards in (loc.~cit) shows that $\Ann(M)=|\lambda\cup\mu|$. Thus,
\begin{align*}
d-\dim_k\cA &=\dim(M)-|\lambda\cup\mu|\\
&=|\lambda|+|\mu|-|\nu|-|\lambda\cup\mu|=|\lambda\cap\mu|-|\nu|.\qedhere
\end{align*}
\end{proof}

\begin{remark}
    Whether or not the inequality $\dim_k\cA\leq d$ holds is independent of the choice of $\phi$, as Proposition \ref{prop:combinatorial-Gerstenhaber-Young-diagram} shows. Correspondingly, via Corollary \ref{cor:combinatorial-mats-Young-diagrams}, whether or not the inequality holds is independent of the choice of the $b_i$ and $c_i$.
\end{remark}

We end this section by specializing to our case of interest. 

\begin{definition}\label{def:tower}
    We say $(\lambda,\nu,b)$ is a \emph{tower} if $\lambda$ is a $3$-dimensional Young diagram, $\nu=(\nu_1,\dots,\nu_r)$ where each $\nu_i$ is an abstract skew shape of the form $\{(0,0,z)\mid 0\leq z<n_i,\, z\in\N\}$, and $b_i\in\Z^n$ satisfying properties \ref{cor:combinatorial-mats-Young-diagrams::propertystart} and \ref{cor:combinatorial-mats-Young-diagrams::propertyideal} of Corollary \ref{cor:combinatorial-mats-Young-diagrams}. We say $(\lambda,\mu,\nu,b,c)$ is a \emph{compatible tower} if $(\lambda,\nu,b)$ and $(\mu,\nu,c)$ are towers. We define $|\nu|:=\sum_{i=1}^r|\nu_i|$ and frequently refer to $\nu_i$ as a $1\times1\times n_i$ shape.
\end{definition}

By Corollary \ref{cor:combinatorial-mats-Young-diagrams} and Proposition \ref{prop:combinatorial-Gerstenhaber-Young-diagram}, we see Theorem \ref{thm:main-intro} is equivalent to

\begin{theorem}\label{thm:main}
If $(\lambda,\mu,\nu,b,c)$ is a compatible tower, then
\[
|\nu|\leq |\lambda\cap\mu|.
\]
\end{theorem}

The remainder of this paper is devoted to proving Theorem \ref{thm:main}.

\begin{remark}\label{rmk:2d-projection-of-3d-Young-diagram}
    It is often useful to think of a three-dimensional Young diagram $\lambda$ in terms of two-dimensional data. Let $\pi\colon\N^3\to\N^2$ be the projection onto the first two coordinates. Then specifying $\lambda$ is equivalent giving a function $H_\lambda\colon\N^2\to\N$ with the property that $H_\lambda(v)\geq H_\lambda(w)$ whenever $v\leq w$ in the poset partial order. Specifically, this equivalence is given by letting $H_\lambda(v)=|\lambda\cap\pi^{-1}(v)|$. We sometimes refer to $H_\lambda(v)$ as the \emph{height of $\lambda$ over $v$}.
\end{remark}




\section{Restricting the shape of counter-examples:~scaffolded towers}

We introduce the following partial order which will play a crucial role in our paper.

\begin{definition}
\label{def:Partial-Order-on-Examples}
    We define a partial order on compatible towers as follows: $(\lambda,\mu,\nu,b,c)\leq(\lambda',\mu',\nu',b',c')$ if the following hold:
    \begin{enumerate}
        \item $\lambda\subseteq\lambda'$ and $\mu \subseteq \mu'$,
        \item letting $\nu=(\nu_i\mid i\in I)$ and $\nu'=(\nu'_i\mid i\in J)$ there exists an injection $\iota\colon I \hookrightarrow J$ such that $\nu_i\subseteq\nu_{\iota(i)}'$,
        \item\label{def:partial-order::counterex} 
        $|\lambda\cap\mu|  - |\nu| \leq |\lambda'\cap\mu'| - |\nu'|$.
    \end{enumerate}
\end{definition}

\begin{remark}
\label{rmk:existence of minimality}
    By (\ref{def:partial-order::counterex}), if $(\lambda,\mu,\nu,b,c)\leq(\lambda',\mu',\nu',b',c')$ and $(\lambda',\mu',\nu',b',c')$ is a counter-example to the Gerstenhaber problem, then $(\lambda,\mu,\nu,b,c)$ is as well.
\end{remark}

Using Remark \ref{rmk:existence of minimality}, we will reduce the study of potential counter-examples to certain compatible towers, which we call \emph{scaffolded pairs}. We introduce this definition after first recalling the notion of order ideals.

\begin{definition}
    For $T \subseteq \N^3$, we define the \emph{order ideal} of $T$ to be
    \[
    \langle T\rangle := \bigcup_{(x,y,z) \in T}\{(a,b,c) \in \N^3 : (a,b,c) \leq (x,y,z)\}
    \]
    Where $\leq$ is the standard partial order on $\N^3$ given by $(a_1,a_2,a_3) \leq (x_1,x_2,x_3)$ if $a_i \leq x_i$ for all $i$.
\end{definition}
\begin{definition}
\label{def:fastened-shape}
    We say a tower $(\lambda,\nu,b)$ is \emph{scaffolded} if 
    \[
    \lambda = \left\langle \bigcup_{i=1}^r(\nu_i+b_i)\right\rangle
    \]
    where $\nu=(\nu_1,\dots,\nu_r)$.
    We say $(\lambda,\mu,\nu,b,c)$ is a compatible tower is \emph{scaffolded} (or a \emph{compatible scaffolded tower}) if both $(\lambda,\nu,b)$ and $(\mu,\nu,c)$ are scaffolded.
\end{definition}

\begin{example}
\label{ex:fastened-example}
    Let $e_i$ denote the $i$-th standard basis vector of $\Z^n$. Let $\nu_1=\nu_2=\nu_3 = \{(0,0,0)\}$ and $b_i=2e_i$. Then the tower $(\lambda,\nu,b)$ depicted below on the left is not scaffolded whereas the tower on the right is.
    \begin{center}
        \includegraphics[width = 0.4\textwidth]{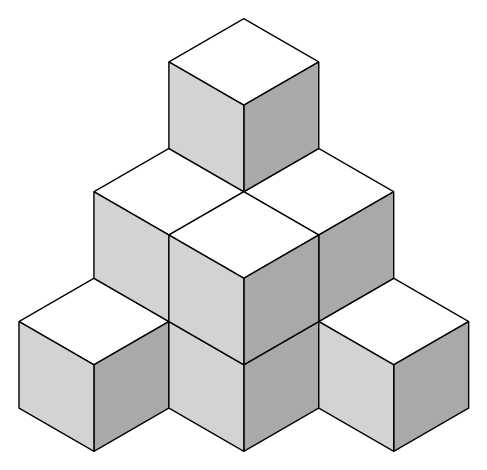}
        \includegraphics[width = 0.4\textwidth]{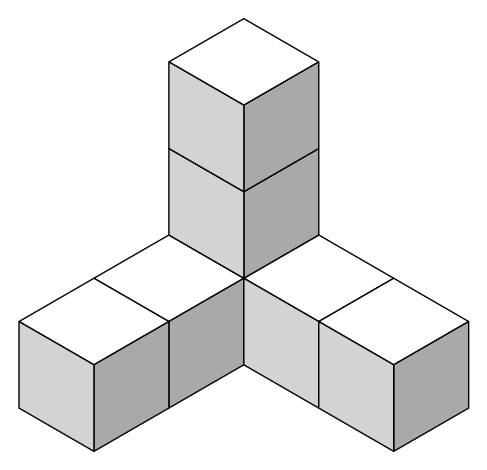}
    \end{center}
\end{example}


The following lemma shows that it suffices to consider compatible scaffolded towers.

\begin{lemma}
\label{lemma:scaffolded towers}
Let $(\lambda,\mu,\nu,b,c)$ be a compatible tower with $\nu=(\nu_1,\dots,\nu_r)$. Letting
    \[
        \lambda' = \left\langle \bigcup_{i=1}^r(\nu_i+b_i)\right\rangle\quad\textrm{and}\quad
        \mu' = \lambda' = \left\langle \bigcup_{i=1}^r(\nu_i+c_i)\right\rangle,
    \]
    we see $(\lambda',\mu',\nu',b,c)$ is scaffolded, and $(\lambda',\mu',\nu',b,c) \leq (\lambda,\mu,\nu,b,c)$.
\end{lemma}
\begin{proof}
We see $(\lambda',\mu',\nu',b,c)$ is a scaffolded by definition. Since $\lambda$ is a three-dimensional Young diagram, if $p\in\lambda$ and $0\leq q\leq p$, then $q\in\lambda$. Since $\nu_i+b_i\subset\lambda$, it follows that the order ideal generated by $\nu_i+b_i$ is contained in $\lambda$, and similarly for $\mu$. Thus, $(\lambda',\mu',\nu',b,c) \leq (\lambda,\mu,\nu,b,c)$.
\end{proof}

\section{Reducing to two-dimensional data:~floor plans}
\label{Sec:Two-dimensional-Picture}

Due to Remark \ref{rmk:existence of minimality} and Lemma \ref{lemma:scaffolded towers}, we have reduced Theorem \ref{thm:main} to the case of compatible scaffolded towers. In this section, we further reduce to two-dimensional data. For this, we introduce the following definition.

\begin{definition}\label{def:floor-plan}
    A \emph{floor plan} is a pair $(P,h)$ of sequences $P=(p_1,\dots,p_r)$, $h=(h_1,\dots,h_r)$ with $p_i\in\N^2$ and $h_i\in\Z^+$. A \emph{compatible floor plan} is a triple $(P,Q,h)$ where $(P,h)$ and $(Q,h)$ are floor plans.
\end{definition}

We have a function
\[
\cF\colon\{\textrm{scaffolded\ towers}\}\to \{\textrm{floor\ plans}\}
\]
defined as follows. If $(\lambda,\nu,b)$ is a scaffolded tower with $\nu=(\nu_1,\dots,\nu_r)$ and $\pi\colon\N^3\to\N^2$ is given by $\pi(x,y,z)=(x,y)$, then $\cF(\lambda,\nu,b)=((p_i)_i,(h_i)_i)$, where $h_i=|\nu_i|$ and $p_i=\pi(b_i)$. Similarly, we have an induced function 
\[
\cF\colon\{\textrm{compatible\ scaffolded\ towers}\}\to \{\textrm{compatible\ floor\ plans}\}
\]
which we also denote by $\cF$. The notation $(P,h)$ is chosen since one should think of $P$ as the positions of the $\nu_i$ and $h$ as the associated heights.

Our next goal is to construct suitable sections of the maps denoted by $\cF$. For this, we need some preliminary definitions. A \emph{North-East path} 
is a sequence $\gamma:=(q_1,q_2,\dots,q_m)$ with $q_i\in\N^2$ and $q_{i+1}-q_i\in\{(1,0),(0,1)\}$. Note that we do not require $q_1=(0,0)$. We say $\gamma$ \emph{originates at} $q_1$.

\begin{definition}\label{def:score-path}
    Given a floor plan $(P,h)$ with $P=(p_1,\dots,p_r)$ and $h=(h_1,\dots,h_r)$, the \emph{score} of a North-East path $\gamma=(q_1,\dots,q_m)$ is
    \[
    \score_{(P,h)}(\gamma):=\sum_{p_i\in\{q_1,\dots,q_m\}}h_i.
    \]
    The \emph{max score} of a lattice point $q\in\N^2$ is
    \[
    \max\score_{(P,h)}(q):=\max\{\score_{(P,h)}(\gamma)\mid \gamma\textrm{\ originates\ at\ } q\}.
    \]
    Any North-East path $\gamma$ originating at $q$ for which $\max\score_{(P,h)}(q)=\score(\gamma)$ is referred to as a \emph{winning path} for $q$. When $(P,h)$ is understood from context, we suppress it in the notation for $\score$ and $\max\score$.
\end{definition}

\begin{definition}\label{def:minimal-realization-of-floor-plan}
    Given a floor plan $(P,h)$, its \emph{minimal realization} $\cT(P,h)$ is the tower $(\lambda,\nu,b)$ defined as follows. Let $\nu_i$ be a $1\times 1\times h_i$ shape and let $\lambda$ be the three-dimensional Young diagram whose corresponding function $H_\lambda\colon\N^2\to\N$ is given by $H_\lambda(q):=\max\score_{(P,h)}(q)$, see Remark \ref{rmk:2d-projection-of-3d-Young-diagram}. We let $b_i=H_\lambda(p_i)-h_i$, where $P=(p_i)_i$.
    
    Similarly, given a compatible floor plan $(P,Q,h)$, its \emph{minimal realization} $\cT(P,Q,h)$ is $(\lambda,\mu,\nu,b,c)$ where $(\lambda,\nu,b)=\cT(P,h)$ and $(\mu,\nu,c)=\cT(Q,h)$.
\end{definition}

\begin{example}\label{ex:minimal-realization}
Consider the floor plan $(P,h)$ where the set $P$ consists of the labeled boxes and the number in each box denotes the corresponding value of $h$.
\begin{center}
\begin{ytableau}
3 \\
\none & 2  \\
\none & \none & \none & 5 \\
2 & \none & \none &  3\\
\none & \none & \none & 1& 1 & 4
\end{ytableau}
\end{center}
Let $(\lambda,\nu,b)=\cT(P,h)$. Then $\lambda$ is given by the function $H_\lambda$ whose values are depicted below.
\begin{center}
\begin{ytableau}
3 \\
3 & 2  \\
5 & 5 & 5 & 5 \\
10 & 8 & 8 & 8\\
10 & 9 & 9 & 9 & 5 & 4
\end{ytableau}
\end{center}
\end{example}

\begin{proposition}\label{prop:minimal-tower-section}
    The minimal realizations $\cT$ yield sections of the two maps denoted by $\cF$. Furthermore, if $(\lambda,\mu,\nu,b,c)$ is scaffolded, then
    \[
    \cT(\cF(\lambda,\mu,\nu,b,c))\leq (\lambda,\mu,\nu,b,c)
    \]
    with respect to the partial order in Definition \ref{def:Partial-Order-on-Examples}.
\end{proposition}

\begin{remark}\label{rmk:stronger-minimal-tower-section}
    In fact, the proof of Proposition \ref{prop:minimal-tower-section} shows the stronger statement that if $(\lambda,\nu,b)$ is a scaffolded tower, then $(\lambda',\nu,b')=\cT(\cF(\lambda,\nu,b))$ satisfies $\lambda'\subseteq\lambda$.
\end{remark}

\begin{proof}
Let $(P,h)$ be a floor plan with $P=(p_i)_i$. We first show that $(\lambda,\nu,b):=\cT(P,h)$ is scaffolded. Let $q\in\N^2$ with $m:=\max\score_{(P,h)}(q)>0$.
We must show $q+(0,0,m)$ is in the order ideal generated by the $(0,0,h_i-1) + b_i$.

For this, let $\gamma=(q_1,q_2,\dots,q_s)$ be a winning path originating at $q_1:=q$. Since $m>0$, there exists $q_\ell\in P$. Without loss of generality, $q_\ell=p_1$ and $q_i\notin P$ for $i<\ell$. Then $(q_\ell,q_{\ell+1},\dots,q_s)$ is a winning path originating at $q_\ell$; indeed, if $\gamma'$ is a path originating at $q_\ell$ with a strictly larger score, then $(q_1,\dots,q_{\ell-1})$ concatenated with $\gamma'$ would yield a strictly larger score than $\gamma$. Thus, $q+(0,0,m)$ is in the order ideal generated by $q_\ell+(0,0,m)$, and this latter element is, by construction, in $(0,0,h_1-1) + b_1$. We have therefore shown $\cT(P,h)$ is scaffolded.

Next, let $(\lambda,\nu,b)$ be a scaffolded tower, $(P,h):=\cF(\lambda,\nu,b)$ the associated floor plan, and $(\lambda',\nu,b')=\cT(P,h)$. We claim that if $q\in\N^2$, we have $H_\lambda(q)\geq\score_{(P,h)}(\gamma)$ for any North-East path $\gamma$ originating at $q$. Let $\gamma=(q_1,\dots,q_s)$ and assume without loss of generality that $q_{i_j}=p_j$ for $i_1<\dots<i_\ell$, and for $t\neq i_j$ we have $q_t\notin P$. We know $\nu_j+b_j\subset\lambda$ and by property \ref{cor:combinatorial-mats-Young-diagrams::propertyideal} of Corollary \ref{cor:combinatorial-mats-Young-diagrams}, $w+b_j\notin\lambda$ for all $w>0$. Therefore, 
\[
H_\lambda(q_{i_j})\geq |\nu_j|+H_\lambda(q_{i_j+1})=h_j+H_\lambda(q_{i_j+1}).
\]
Thus, 
\[
H_\lambda(q)\geq\sum_j h_j=\score_{(P,h)}(\gamma).
\]
 As a result, $\lambda'\subseteq\lambda$. Note that since $\nu$ remains unchanged under the operation $\cT\circ\cF$, the proposition follows.
\end{proof}

\section{Constraints on the border of a floor plan}
\label{Subsec:2D-min}

In this section, we further reduce Theorem \ref{thm:main} to the study of floor plans whose borders are highly constrained. Throughout this section, we use the following notation. Let $e_1,e_2,e_3$ be the standard basis vectors of $\Z^3$. For any $p\in\N^3$, we let $x(p)$, $y(p)$, and $z(p)$ denote the $x$, $y$, and $z$ coordinates of $p$. If $(P,h)$ is a floor plan with $P=(p_i)_i$, we will sometimes write $p\in P$ to mean $p=p_i$ for some $i$.

\begin{definition}\label{def:minimal-floor-plan}
Let $(P,h)$ and $(P',h)$ be floor plans\footnote{Note that the second coordinates of these floor plans are the same.} and let $(\lambda,\nu,b)=\cT(P,h)$ and $(\lambda',\nu,b')=\cT(P',h)$. We write
\[
(P',h)\leq(P,h)\quad\textrm{if}\quad\lambda'\subset\lambda.\footnote{This is equivalent to $\cT(P',h)\leq\cT(P,h)$ with respect to the partial ordering in Definition \ref{def:Partial-Order-on-Examples}.}
\]
Similarly, given compatible floor plans $(P,Q,h)$ and $(P',Q',h)$, we write
\[
(P,Q,h)\leq(P',Q',h)\quad\textrm{if}\quad\cT(P,Q,h)\leq\cT(P',Q',h)
\]
with respect to the partial ordering in Definition \ref{def:Partial-Order-on-Examples}. We say $(P,h)$, respectively $(P,Q,h)$, is \emph{minimal} if it is minimal with respect to these partial orders.
\end{definition}

\begin{definition}\label{def:border}
    Let $(P,h)$ be a floor plan. We define the \emph{support} of $(P,h)$ to be
    \[
    \supp(P,h):=\{q\in\N^2\mid \max\score_{(P,h)}(q)>0\}.
    \]
    The \emph{border} $B(P,h)$ is then defined as all $q\in\supp(P,h)$ such that $q+e_1$ or $q+e_2$ is not in $\supp(P,h)$.
\end{definition}
\begin{remark}
    Note that $\supp(P,h)$ is the projection of $\cT(P,h)$ onto the $xy$-plane.
\end{remark}


Our first goal is to prove:

\begin{proposition}\label{prop:complete-max-trails-exist}
    Let $(P,h)$ be a minimal floor plan. Then
    \[
    B(P,h)\subseteq P.
    \]
    In particular, if $(P,Q,h)$ is a minimal compatible floor plan, then $B(P,h)\subseteq P$ and $B(Q,h)\subseteq Q$.
\end{proposition}

We prove this proposition after a preliminary result.

\begin{lemma}\label{l:shrinking-path1}
    Let $(P,h)$ be a floor plan. Suppose there exists $i$ such that
    \begin{enumerate}
        \item $x(p_i)>0$ and
        \item for all $j$ with $x(p_j)=x(p_i)-1$, we have $y(p_j)<y(p_i)$.
    \end{enumerate}
    Then letting $p'_i=p_i-e_1$ and $p'_k=p_k$ for all $k\neq i$, we have $(P',h)<(P,h)$ where $P'=(p'_j)_j$.
\end{lemma}
\begin{proof}
Note that $\max\score_{(P',h)}(p_i)<\max\score_{(P,h)}(p_i)$ and that $\max\score_{(P',h)}(p)\leq\max\score_{(P,h)}(p)$ if $p<p_i$ with $x(p)=x(p_i)$. Further note that for all $p$ with $p\not\leq p_i$, we have $\max\score_{(P',h)}(p)=\max\score_{(P,h)}(p)$. So, to prove $(P',h)<(P,h)$, it suffices to show $\max\score_{(P',h)}(v)\leq\max\score_{(P,h)}(v)$ for all $v\leq p_i-e_1$.

For this, consider a North-East path $\gamma$ which contains $p'_i$. Let $\gamma=(q_1,\dots,q_s)$ with $q_\ell=p'_i$. Say $x(q_j)=x(p'_i)$ for $\ell\leq j\leq m$ and that $x(q_{m+1})>x(p'_i)$; this implies $q_{m+1}=q_m+e_1$. Let
\[
\gamma'=(q_1,\dots,q_\ell,q_\ell+e_1,\dots,q_m+e_1,q_{m+2},\dots,q_s).
\]
By hypothesis, $q_j\notin P$ for $\ell\leq j\leq m$, so
\[
\score_{(P',h)}(\gamma)\leq\score_{(P',h)}(\gamma')=\score_{(P,h)}(\gamma')
\]
which proves the result.
\end{proof}

\begin{proof}[{Proof of Proposition \ref{prop:complete-max-trails-exist}}]
Let $v\in B\setminus P$. Note that we cannot have $v+e_1\notin\supp(P)$ and $v+e_2\notin\supp(P)$ since this implies then $\max\score_{(P,h)}(v)=\score_{(P,h)}(\gamma)=0$, where $\gamma$ is the singleton path $(v)$. Thus, without loss of generality, $v+e_1\in\supp(P)$ and $v+e_2\notin\supp(P)$. As a result, $\max\score_{(P,h)}(v)$ is the sum of the $h_i$ for all $i$ with $y(p_i)=y(v)$. Since this quantity is non-zero, we may assume without loss of generality that $y(v)=y(p_1)$, $x(v)<x(p_1)$, and there are no $p\in P$ such that $y(v)=y(p)$ and $x(v)<x(p)<x(p_1)$. Let $P'=(p'_i)_i$ where $p'_1=p_1-e_1$ and $p'_i=p_i$ for $i\neq1$. Then by Lemma \ref{l:shrinking-path1}, $(P',h)<(P,h)$, showing that $(P,h)$ is not minimal.
\end{proof}

%

We end this section with some further useful properties of minimal compatible floor plans.

\begin{lemma}\label{l:no-overlap}
    If $(P,Q,h)$ is a minimal compatible floor plan, then $P\cap Q=\varnothing$.
\end{lemma}
\begin{proof}
    Let $P=(p_i)_i$ and $Q=(q_i)_i$. 
    Let $(\lambda,\mu,\nu,b,c)=\cT(P,Q,h)$. We show that if $P\cap Q$ is non-empty, then $(\lambda,\mu,\nu,b,c)$ is not minimal. Without loss of generality, $p_i=q_j$ and $\max\score_{(P,h)}(p_i)\geq\max\score_{(Q,h)}(q_j)$. Let $h'_j=h_j-1$ and $h'_m=h_m$ for all $m\neq j$. Letting $(\lambda,\mu',\nu',b',c')=\cT(P,Q,h')$, we see
    \[
|\lambda\cap\mu'|\leq|\lambda\cap\mu|-1
    \]
    and so $(\lambda,\mu',\nu',b',c')<(\lambda,\mu,\nu,b,c)$.
\end{proof}

It follows immediately from Proposition \ref{prop:complete-max-trails-exist} and Lemma \ref{l:no-overlap} that we have:

\begin{corollary}
\label{cor:inclusion of ribbons}
    If $(P,Q,h)$ is a minimal compatible floor plan, then 
    \[
    \supp(P)\subset\supp(Q)\quad\textrm{or}\quad \supp(Q)\subset\supp(P).
    \]
\end{corollary}

%

\section{Proof of Theorem \ref{thm:main-intro}}

We turn now to Theorem \ref{thm:main-intro}. As shown in \S\ref{sec:review-combinatorial-approach}, this is equivalent to proving Theorem \ref{thm:main}.

\begin{proof}[{Proof of Theorem \ref{thm:main}}]
Assume there exists a counter-example to the theorem. By Remark \ref{rmk:existence of minimality}, Lemma \ref{lemma:scaffolded towers}, and Proposition \ref{prop:minimal-tower-section}, we may assume $(\lambda,\mu,\nu,b,c)=\cT(P,Q,h)$ with $(P,Q,h)$ a minimal compatible floor plan. By Corollary \ref{cor:inclusion of ribbons}, we may assume without loss of generality that $\supp(P)\subset\supp(Q)$.

Let $P=(p_1,\dots,p_r)$ and $Q=(q_1,\dots,q_r)$. Reindexing, we may assume there exists $N_1$ such that $h_i=1$ and $q_i\in\supp(Q)$ is maximal (in the partial ordering on $\N^2$) if and only if $i>N_1$. We may further assume there exists $N_0\leq N_1$ such that $q_i\in\supp(Q)$ is maximal if and only if $i>N_0$. Let
\[
P'=(p_i\mid i\leq N_1),\quad\textrm{and}\quad Q'=(q_i\mid i\leq N_1), \quad\textrm{and}\quad h'=(h_1,\dots,h_{N_0},h_{N_0+1}-1,\dots,h_{N_1}-1);
\]
in other words, $h'$ decreases the value of $h$ at all maximal elements of $\supp(Q)$ and deletes any indices which now have value $0$. Let $(\lambda',\mu',\nu',b',c')=\cT(P',Q',h')$. We claim $(\lambda',\mu',\nu',b',c')<(\lambda,\mu,\nu,b,c)$, contradicting minimality of $(\lambda,\mu,\nu,b,c)$.

To see this, first note that 
\[
\max\score_{(P',h')}(p)\leq\max\score_{(P,h)}(p)
\]
for all $p\in\supp(P)$ and that
\[
\max\score_{(P',h')}(p_i)\leq\max\score_{(P,h)}(p_i)-1
\]
for all $N_0<i\leq N_1$. Thus, $\lambda'\subsetneq\lambda$. Furthermore, $B(Q,h)\subset Q$ by Proposition \ref{prop:complete-max-trails-exist}; it follows that for all $q\in\supp(Q)$, any winning path $\gamma$ for $(Q,h)$ originating at $q$ must contain a (necessarily unique) maximal element of $\supp(Q)$. Since the value of $h'$ is one less than the value of $h$ at all maximal elements of $\supp(Q)$, we see
\[
\max\score_{(Q',h')}(q)\leq\max\score_{(Q,h)}(q)-1
\]
for all $q\in\supp(Q)$. In particular, $\mu'\subsetneq\mu$.

To show $(\lambda',\mu',\nu',b',c')<(\lambda,\mu,\nu,b,c)$, it therefore remains to prove
\begin{equation}\label{eqn:thm-main}
|\lambda' \cap \mu'| - \sum_j |\nu_j'| \leq |\lambda \cap \mu| - \sum_i |\nu_i|.
\end{equation}
For a three-dimensional Young diagram $\epsilon$ and $v\in\N^2$, let $\height_\epsilon(v)$ denote the number of boxes of $\epsilon$ living over $v$. 
For $i>N_0$, we see
\begin{align*}
\height_{\lambda'\cap\mu'}(p_i) &=\min(\max\score_{(P',h')}(p_i),\max\score_{(Q',h')}(p_i))\\
&\leq \min(\max\score_{(P,h)}(p_i),\max\score_{(Q,h)}(p_i))-1=\height_{\lambda\cap\mu}(p_i)-1.
\end{align*}
Thus,
\[
|\lambda'\cap\mu'| \ =\ \sum_{v\in \N^2}\height_{\lambda'\cap\mu'}(v)
\ \leq\ \sum_{v\in \N^2}\height_{\lambda\cap\mu}(v) - (r-N_0)\ \leq\ |\lambda\cap\mu|-(r-N_0).
\]
Observing that $\sum_i|\nu_i|-\sum_j|\nu'_j|=r-N_0$, we see \eqref{eqn:thm-main} holds.
\end{proof}

\bibliography{Gerstenhaber}
\bibliographystyle{amsalpha}

\end{document}